\documentclass[a4paper,11pt,leqno]{amsart}
\usepackage{amsmath}
\usepackage{amssymb}
\usepackage{amsthm}
\usepackage{indentfirst}
\usepackage[english]{babel}
\usepackage[T1]{fontenc}  
\usepackage{breqn}
\usepackage{cite}    
\usepackage{mathrsfs}

\theoremstyle{plain} 
\newtheorem{thm}{Theorem}[section] 
 
\newtheorem{lem}[thm]{Lemma} 
\newtheorem{prop}[thm]{Proposition} 

\theoremstyle{definition}

\numberwithin{equation}{section}

\title[Short intervals containing a prescribed number of primes]
      {Short intervals containing a prescribed number of primes}

\author[D. Mastrostefano]{Daniele Mastrostefano}

\email{danymastro93@hotmail.it}

\begin{document}

\begin{abstract}
We will prove that for every $m\geq 0$ there exists an $\varepsilon=\varepsilon(m)>0$ such that if $0<\lambda<\varepsilon$ and $x$ is sufficiently large in terms of $m$, then 
$$|\lbrace n\leq x: |[n,n+\lambda\log n]\cap \mathbb{P}|=m\rbrace|\gg_{m} \frac{x}{\log x}.$$
\end{abstract}

\maketitle

\section{Introduction}
\label{sec:1}
Let $\mathbb{P}$ denote the set of prime numbers and fix $\lambda>0$ a real number and $m$ a nonnegative integer. Recently, Maynard has shown, in his breakthrough paper \cite{M2} on bounded gaps between primes, that there exist infinitely many intervals $[n,n+\lambda\log n]$, with $n\in\mathbb{N}$, containing at least $m$ primes. More recently, Maynard proved (see the proof of \cite[Theorem 3.3]{M1}) that in fact there exists a positive proportion of such kind of intervals, when $\lambda$ is a small parameter depending on $m$. In particular, he showed that for every $m\geq 0$ there exists an $\varepsilon=\varepsilon(m)>0$ such that if $0<\lambda<\varepsilon$ and $x$ is sufficiently large in terms of $m$, then 
$$|\lbrace n\leq x: |[n,n+\lambda\log n]\cap \mathbb{P}|\geq m\rbrace|\gg_{m} x.$$
However, these statements do not preclude the possibility that there are choices of $\lambda$ and $m$ for which the intervals $[n,n+\lambda\log n]$ contain exactly $m$ primes, for at most finitely many $n$. Anyway, this was proven to be not the case by Freiberg, who showed the following more precise result.
\begin{thm}
 \label{thm:1.1}
For any positive real number $\lambda$ and any nonnegative integer $m$, we have
\begin{equation}
 \label{eq:1.1}
|\lbrace n\leq x: |[n,n+\lambda\log n]\cap \mathbb{P}|= m\rbrace|\geq x^{1-\varepsilon(x)},
\end{equation}
if $x$ is sufficiently large in terms of $\lambda$ and $m$, and $\varepsilon(x)$ is a certain function that tends to zero as $x$ tends to infinity.
\end{thm}
This is \cite[Theorem 1.1]{F} in which we may take $\varepsilon(x)=(\log\log\log\log x)^{2}/\log\log\log x$. The aim of this paper is to show that, using a careful investigation of the Maynard paper \cite{M1} and the Freiberg paper \cite{F}, we may improve Theorem \ref{thm:1.1}, finding a better lower bound for the size of the set in \eqref{eq:1.1} at the cost to consider only small values of $\lambda$. In fact, we prove the following Theorem.
\begin{thm} 
 \label{thm:1.2}
For any nonnegative integer $m$, there exists an $\varepsilon=\varepsilon(m)>0$ such that for every $0<\lambda<\varepsilon$ we have
\begin{equation}
 \label{eq:1.2}
|\lbrace n\leq x: |[n,n+\lambda\log n]\cap \mathbb{P}|=m\rbrace|\gg_{m} \frac{x}{\log x},
\end{equation}
if $x$ is sufficiently large in terms of $m$.
\end{thm}
From an heuristic point of view, we expect a positive proportion of such short intervals containing a prescribed number of primes, i.e.
$$|\lbrace n\leq x: |[n,n+\lambda\log n]\cap \mathbb{P}|= m\rbrace|\gg_{m,\lambda} x,$$
if $x$ is sufficiently large in terms of $\lambda$ and $m$. More precisely, we conjecture that 
$$ |\lbrace n\leq x: |[n,n+\lambda\log n]\cap \mathbb{P}|= m\rbrace|\backsim \frac{\lambda^{m} e^{-\lambda}}{m!}x,$$
for every $\lambda$ and $m$. We refer the reader to the expository article \cite{S} of Soundararajan for further discussions on these fascinating statistics. In order to prove Theorem \ref{thm:1.2} we use a combination of ideas present in \cite{F} and \cite{M1}. The Freiberg's approach was to construct a special admissible set of linear forms, by using an Erd\H{o}s--Rankin type construction \cite[Lemma 3.3]{F}, to which apply the Maynard result \cite[Theorem 3.1]{M1}, finding many disjoint intervals containing at least $m$ primes. The conclusion now follows by using a sliding process to detect several different intervals containing exactly $m$ primes. In this paper, we use the Maynard sieve method introduced in \cite{M1}, in place of the Erd\H{o}s--Rankin construction, finding a large proportion of particular intervals containing at least $m$ primes to which apply the combinatorial process of Freiberg. Unfortunately, avoiding the use of the Erd\H{o}s--Rankin construction limits the choice of $\lambda$, which can be taken only very small.
\section{Notations and preliminaries}
Throughout, $\mathbb{P}$ denotes the set of all primes, $\mathbf{1}_{\mathbb{P}}: \mathbb{N}\rightarrow \lbrace 0,1\rbrace$ the indicator function of $\mathbb{P}\subset \mathbb{N}$ and $p$ a prime. As usual, $\varphi$ will denote the Euler totient function, $\mu$ the Moebius function and $\omega(n)$ is the counting function of the number of different prime factors of a positive integer $n$. We let $(m,n)$ be the greatest common divisor of integers $n$ and $m$. We will always denote with $x$ a sufficiently large real number. By $o(1)$ we mean a quantity that tends to $0$ as $x$ tends to infinity. The expressions $A=O(B), A\ll B, B\gg A$ denote that $|A|\leq c|B|$, where $c$ is some positive (absolute, unless stated otherwise) constant.\\
In the following we will always consider admissible $k$-tuples of linear forms $\lbrace gn+h_{1},...,gn+h_{k}\rbrace$, where $0< h_{1}<h_{2}<...<h_{k}<\lambda\log x$, $k\geq 0$ a fixed sufficiently large integer and $g$ a fixed positive integer, coprime with $B$, squarefree and such that $\log x<g\leq 2\log x$. Here, $B=1$ or $B$ is a prime with $\log\log x^{\eta}\ll B\ll x^{2\eta},$ where we put $\eta=c/500k^{2}$ with $0<c<1$. For example, we may take $g$ as a prime number in the interval $[\log x, 2\log x],$ with $g\neq B$. As usual, a finite set $\lbrace L_{1},...,L_{k}\rbrace$ of linear functions is admissible if the set of solutions modulo $p$ to $L_ {1}(n)\cdots L_{k}(n)\equiv 0\pmod{p}$ does not form a complete residue system modulo $p$, for any prime $p$. In our case, in which $L_{i}(n)=gn+h_{i}$, for every $i=1,...,k$, we may infer that the set $\lbrace L_{1},...,L_{k}\rbrace$ is admissible if and only if the set $\lbrace h_{1},...,h_{k}\rbrace$ it is, in the sense that the elements $h_{1},...,h_{k}$ do not cover all the residue classes modulo $p$, for any prime $p$.
We quote \cite[Proposition 6.1]{M1}, adapted to our situation:
\begin{prop}
\label{prop:2.1} Let $\mathcal{H}=\lbrace h_{1},...,h_{k}\rbrace$ be an admissible set with $0<h_{1}<...<h_{k}<\lambda\log x$, with $k, B$ positive integers and $x,\lambda$ positive real numbers. Suppose $1\leq B\leq x^{2\eta}$, with $\eta=c/500k^{2}$ and $c>0$. Let $R=x^{\frac{1}{24•}}$, $k(\log\log x)^{2}/(\log x)\leq \rho\leq 1/80$ and $g$ a positive integer in the interval $[\log x, 2\log x]$, squarefree and coprime with $B$. There is a constant $C>0$ such that the following holds. If $k\geq C$, then there is a choice of nonnegative weights $w_{n}=w_{n}(\mathcal{H})$ satisfying 
\begin{equation}
\label{eq:2.1}
w_{n}\ll (\log R)^{2k}\prod_{i=1}^{k}\prod_{p|gn+h_{i}, p\nmid B}4
\end{equation}
for which we have, for every integer $h\in[1, 5\lambda\log x]$, that
\begin{equation}
\label{eq:2.2}
\sum_{x<n\leq 2x}w_{n}=\left(1+O\left(\frac{1}{(\log x)^{1/10}}\right)\right)\frac{B^{k}}{\varphi(B)^{k}•}\mathfrak{S}_{B}(\mathcal{H})x(\log R)^{k}I_{k},
\end{equation}
\begin{equation}
\label{eq:2.3}
\sum_{x<n\leq 2x}\mathbf{1}_{\mathbb{P}}(gn+h)w_{n}\geq\left(1+O\left(\frac{1}{(\log x)^{1/10}}\right)\right)\frac{B^{k-1}}{\varphi(B)^{k-1}•}\mathfrak{S}_{B}(\mathcal{H})\frac{\varphi(g)}{g•}(\log R)^{k+1}J_{k}\sum_{x<n\leq 2x}\mathbf{1}_{\mathbb{P}}(gn+h)
\end{equation}
$$+O\left(\frac{B^{k}}{\varphi(B)^{k}•}\mathfrak{S}_{B}(\mathcal{H})x(\log R)^{k-1}I_{k}\right),$$
\begin{equation}
\label{eq:2.4}
\sum_{x<n\leq 2x}\mathbf{1}_{S(\rho, B)}(gn+h)w_{n}\ll \rho^{-1}\frac{\Delta_{\mathcal{L}}}{\varphi(\Delta_{\mathcal{L}})}\frac{B^{k+1}}{\varphi(B)^{k+1}•}\mathfrak{S}_{B}(\mathcal{H})x(\log R)^{k-1}I_{k},
\end{equation}
where $\Delta_{\mathcal{L}}=|g|^{k+1}\prod_{i=1}^{k}|h-h_{i}|\neq 0$, with $\mathcal{L}=\lbrace gn+h_{1},...,gn+h_{k}\rbrace$, and where we define 
$$S(\rho, B)=\lbrace n\in \mathbb{N}: p|n\Rightarrow (p>x^{\rho}\ \textrm{or}\ p|B)\rbrace.$$
Finally, we have
\begin{equation}
\label{eq:2.5}
\sum_{x<n\leq 2x}\bigg(\sum_{\substack{p|gn+h\\p<x^{\rho}\\p\nmid B}}1 \bigg)w_{n}\ll \rho^{2}k^{4}(\log k)^{2}\frac{B^{k}}{\varphi(B)^{k}•}\mathfrak{S}_{B}(\mathcal{H})x(\log R)^{k}I_{k}.
\end{equation}
Here $I_{k}, J_{k}$ are quantities depending only on $k$, and $\mathfrak{S}_{B}(\mathcal{L})$ is a quantity depending only on $\mathcal{L}$, and these satisfy
\begin{equation}
\label{eq:2.6}
\mathfrak{S}_{B}(\mathcal{L})\gg \frac{1}{\exp(O(k))},\ \ I_{k}\gg \frac{1}{(2k\log k)^{k}},\ \ J_{k}\gg \frac{\log k}{k•}I_{k}.
\end{equation}
\end{prop}
What is made implicitly in Proposition \ref{prop:2.1} is that, under the hypothesis of the statement, a level of distribution result for linear forms in arithmetic progressions holds. In fact, Proposition \ref{prop:2.1} depends heavily on what is called Hypothesis 1 in \cite{M1}. We will see that our admissible set of linear functions satisfies a such kind of result in the following Theorem.
\begin{thm}
\label{thm:2.2}
Consider $k$ sufficiently large, and put $\eta=c/500k^{2}$ with $c>0$. Let $\mathcal{L}=\lbrace gn+h_{1},...,gn+h_{k}\rbrace$ be an admissible set of linear forms with $0< h_{1}<h_{2}<...<h_{k}<\lambda\log x$ and $g$ a positive integer, coprime with $B$, squarefree and such that $\log x<g\leq 2\log x$, with $B$ a positive integer. Then $B$ and $c$ may be chosen so that the following holds once $x$ is large enough in terms of $k$. For each $L\in\mathcal{L}$, 
\begin{equation}
\label{eq:2.7}
\frac{\varphi(B)}{B•}\frac{\varphi(g)}{g•}\sum_{x<n\leq 2x}\mathbf{1}_{\mathbb{P}}(L(n))>\frac{x}{2\log x•},
\end{equation}
and
\begin{equation}
\label{eq:2.8}
\sum_{\substack{q\leq x^{1/8}\\(q,B)=1}}\max_{(L(a),q)=1)}\bigg|\sum_{\substack{x<n\leq 2x\\ n\equiv a\pmod{q}}}\mathbf{1}_{\mathbb{P}}(L(n))-\frac{\varphi(g)}{\varphi(gq)}\sum_{x<n\leq 2x}\mathbf{1}_{\mathbb{P}}(L(n))\bigg|\ll \frac{\sum_{x<n\leq 2x}\mathbf{1}_{\mathbb{P}}(L(n))}{(\log x)^{100k^{2}}}.
\end{equation}
Moreover, $c<1$ and either $B=1$ or $B$ is a prime satisfying $\log\log x^{\eta}\ll B\ll x^{2\eta}$.
\end{thm}
\begin{proof}
This is precisely \cite[Lemma 3.2]{F}, adapted to our situation. In particular, the main difference is in the choice of $g$. Indeed, here we chose $g\in [\log x, 2\log x]$ whereas in \cite{F} was chosen $g=\prod_{p\leq \log x^{\eta}}p$. Anyway, the structure of the proof of the Theorem \ref{thm:2.2} is exactly the same of \cite[Lemma 3.2]{F} and the computations change only in a few points, in which it is very easy to verify that they do not affect the stated results \eqref{eq:2.7} and \eqref{eq:2.8}. For the aforementioned reasons, we omit the details of the proof.
\end{proof}
\section{The Maynard sieve method}
Under the hypothesis and notations of Proposition \ref{prop:2.1}, we define the following sum:
\begin{equation}
\label{eq:3.1}
S=\sum_{\substack{\mathcal{H}=\lbrace h_{1},...,h_{k}\rbrace\ \textrm{admissible}\\ 0< h_{1}<h_{2}<...<h_{k}<\lambda\log x}}\sum_{x<n\leq 2x}S(\mathcal{H},n),
\end{equation}
where 
$$S(\mathcal{H},n)=\bigg( \sum_{i=1}^{k} \textbf{1}_{\mathbb{P}}(gn+h_{i}) -(m-1)-k\sum_{i=1}^{k}\sum_{\substack{p|gn+h_{i}\\ p\leq x^{\rho}, p\nmid B}}1 -k\sum_{\substack{ h\leq 5\lambda\log x\\ (h,g)=1\\h\not\in\mathcal{H}}} \textbf{1}_{S(\rho, B)}(gn+h)\bigg)w_{n}(\mathcal{H}),$$
with $m$ a nonnegative integer. We note that
\begin{equation}
\label{eq:3.2}
S\ll k(\log x)^{2k} \exp\left(O\left(\frac{k}{\rho•}\right)\right)\sum_{\substack{\mathcal{H}=\lbrace h_{1},...,h_{k}\rbrace\ \textrm{admissible}\\ 0< h_{1}<h_{2}<...<h_{k}<\lambda\log x}} \sum_{x<n\leq 2x} \mathbf{1}_{S(\mathcal{H}, n)>0},
\end{equation}
uniformly in $n$ and $\mathcal{H}$, by \eqref{eq:2.1} and the fact that
\begin{equation} 
\label{eq:3.3}
k\sum_{i=1}^{k}\sum_{\substack{p|gn+h_{i}\\ p\leq x^{\rho}, p\nmid B}}1=0,
\end{equation} 
if $S(\mathcal{H}, n)>0$. In order to have $S(\mathcal{H}, n)>0$ we must have
\begin{equation}
\label{eq:3.4}
\sum_{i=1}^{k} \textbf{1}_{\mathbb{P}}(gn+h_{i}) -(m-1)>0,
\end{equation}
\begin{equation}
\label{eq:3.5}
k\sum_{i=1}^{k}\sum_{\substack{p|gn+h_{i}\\ p\leq x^{\rho}, p\nmid B}}1=0,
\end{equation}
\begin{equation}
\label{eq:3.6}
k\sum_{\substack{ h\leq 5\lambda\log x\\(h,g)=1\\ h\not\in\mathcal{H}}} \textbf{1}_{S(\rho, B)}(gn+h)=0.
\end{equation}
The equation \eqref{eq:3.4} tells us that there are at least $m$ primes among $gn+h_{1},...,gn+h_{k}$; the equation \eqref{eq:3.5} tells us that all the $gn+h_{1},...,gn+h_{k}$ are divided only by primes $p>x^{\rho}$; finally, the equation \eqref{eq:3.6} tells us that all the remaining $gn+h$ are divided by a primes $p\leq x^{\rho}$. In fact, this follows directly by \eqref{eq:3.6} when $(h,g)=1$ and trivially when $(h,g)>1$, because in this case we can certainly find a small prime dividing $gn+h$. In particular, all the $gn+h$ with $h\not\in \mathcal{H}$ cannot be prime numbers and \eqref{eq:3.5} together with \eqref{eq:3.6} imply that the elements $gn+h_{1},...,gn+h_{k}$ are uniquely determined as the integers in $[gn,gn+5\lambda\log x]$ with no prime factors $p\nmid B$ less than $x^{\rho}$. This means that in the intervals $[gn,gn+5\lambda\log x]$ the prime numbers are contained in the set $\lbrace gn+h_{1},...,gn+h_{k}\rbrace$, for an admissible set $\mathcal{H}=\lbrace h_{1},...,h_{k}\rbrace$, and that no $n$ can make a positive contribution from two different admissible sets. So we have proved that:
\begin{equation}
\label{eq:3.7}
S\ll k(\log x)^{2k} \exp\left(O\left(\frac{k}{\rho•}\right)\right)|I(x)|,
\end{equation}
where the set $I(x)$ contains interval of the form  $[gn,gn+5\lambda\log x]$, for $x<n\leq 2x$, with  the property that $|[gn,gn+5\lambda\log x]\cap \mathbb{P}|=|\lbrace gn+h_{1},...,gn+h_{k}\rbrace\cap\mathbb{P}|\geq m$, for a unique admissible set $\mathcal{H}=\lbrace h_{1},...,h_{k}\rbrace$ such that $0<h_{1}<...<h_{k}<\lambda\log x$. We note that the intervals in $I(x)$ are pairwise disjoint. In fact, suppose to consider two intervals $[gn,gn+5\lambda\log x], [gm,gm+5\lambda\log x]$ for two distinct integers $n,m$. Let us assume $n<m$ (the case $n>m$ is similar), then $gn+5\lambda\log x<g(n+1)\leq gm$, if $\lambda<1/5$.\\
Now, we want find a lower bound for $S$. Using \eqref{eq:2.2}--\eqref{eq:2.5}, we find
\begin{equation}
\label{eq:3.8}
S\geq\sum_{\substack{\mathcal{H}=\lbrace h_{1},...,h_{k}\rbrace\ \textrm{admissible}\\ 0< h_{1}<h_{2}<...<h_{k}<\lambda\log x}}\bigg[ (1+o(1))\frac{B^{k-1}}{\varphi(B)^{k-1}•}\mathfrak{S}_{B}(\mathcal{H})(\log R)^{k+1} J_{k}\sum_{i=1}^{k}\frac{\varphi(g)}{g•}\sum_{x<n\leq 2x}\textbf{1}_{\mathbb{P}}(gn+h_{i})
\end{equation}
$$-(m-1)(1+o(1))\frac{B^{k}}{\varphi(B)^{k}•}\mathfrak{S}_{B}(\mathcal{H})x(\log R)^{k} I_{k}+O\left(\rho^{2}k^{6}(\log k)^{2}\frac{B^{k}}{\varphi(B)^{k}•}\mathfrak{S}_{B}(\mathcal{H})x(\log R)^{k} I_{k}\right)$$
$$+O\left(k\frac{B^{k}}{\varphi(B)^{k}•}\mathfrak{S}_{B}(\mathcal{H})x(\log R)^{k-1} I_{k}\right)+O\bigg(\frac{k}{\rho}\frac{B^{k+1}}{\varphi(B)^{k+1}•}\mathfrak{S}_{B}(\mathcal{H})x(\log R)^{k-1} I_{k}\sum_{\substack{h\leq 5\lambda\log x\\(h,g)=1\\ h\not\in \mathcal{H}}}\frac{\Delta_{\mathcal{L}}}{\varphi(\Delta_{\mathcal{L}•})}\bigg)\bigg].$$
By the inequality \eqref{eq:2.7}, we have
$$\frac{\varphi(B)}{B•}\frac{\varphi(g)}{g•}\sum_{i=1}^{k}\sum_{x<n\leq 2x}\textbf{1}_{\mathbb{P}}(gn+h_{i})>\frac{kx}{2\log x•},$$
for every admissible set $\mathcal{H}=\lbrace h_{1},...,h_{k}\rbrace$ such that $0< h_{1}<h_{2}<...<h_{k}< \lambda\log x$. Note that the hypothesis of Theorem \ref{thm:2.2} are satisfied. Moreover, we need the following Lemma.
\begin{lem} 
\label{lem:3.1}
Let us consider an admissible set $\mathcal{L}=\lbrace gn+h_{1},...,gn+h_{k}\rbrace$ with $0< h_{1}<h_{2}<...<h_{k}<\lambda\log x$ and $g$ a positive number in the interval $[\log x, 2\log x]$. If we define $\Delta_{\mathcal{L}}=|g|^{k+1}\prod_{i=1}^{k}|h-h_{i}|$, we have
\begin{equation}
\label{eq:3.9}
\sum_{\substack{h\leq 5\lambda\log x\\ (h,g)=1\\ gn+h\not\in \mathcal{L}}}\frac{\Delta_{\mathcal{L}}}{\varphi(\Delta_{\mathcal{L}•})}\ll \lambda(\log x) (\log k).
\end{equation}
\end{lem}
We leave the proof to the next section. Using these estimates and choosing $\rho=c_{0}k^{-3}(\log k)^{-1}$, with $c_{0}$ a suitable small constant, we find
\begin{equation}
\label{eq:3.10}
S\gg \sum_{\substack{\mathcal{H}=\lbrace h_{1},...,h_{k}\rbrace\ \textrm{admissible}\\ 0< h_{1}<h_{2}<...<h_{k}<\lambda\log x}}\frac{B^{k}}{\varphi(B)^{k}•}\mathfrak{S}_{B}(\mathcal{H})x(\log R)^{k}\bigg[kJ_{k}\frac{\log R}{2\log x}-(m-1)I_{k}+O(c_{0}^{2}I_{k})
\end{equation}
$$+O(k I_{k}(\log R)^{-1})+O\left(\frac{k^{4}(\log k)}{c_{0}} I_{k}(\log R)^{-1}\lambda(\log x)(\log k) \right)\bigg].$$
Now, by \eqref{eq:2.6} we know that $J_{k}\gg \frac{\log k}{k•}I_{k}.$ We should consider $k$ sufficiently large in terms of $m$. For example, we may take $k=C\exp(m^{48})$, with $C>0$ a suitable large constant. Choosing $\lambda\leq\varepsilon$ a small multiple of $k^{-4}(\log k)^{-2}$, we may find
\begin{equation}
\label{eq:3.11}
S\gg \sum_{\substack{\mathcal{H}=\lbrace h_{1},...,h_{k}\rbrace\ \textrm{admissible}\\ 0< h_{1}<h_{2}<...<h_{k}<\lambda\log x}}\frac{B^{k}}{\varphi(B)^{k}•}\mathfrak{S}_{B}(\mathcal{H})x(\log R)^{k}I_{k}.
\end{equation}
By the estimates \eqref{eq:2.6} we know that $I_{k}\gg (2k\log k)^{-k}$ and that $\mathfrak{S}_{B}(\mathcal{L})\gg \exp(-C_{1}k)$, for a certain $C_{1}>0$. Finally, we may certainly use $\frac{B^{k}}{\varphi(B)^{k}•}\geq 1$. Inserting these in \eqref{eq:3.11}, we obtain
\begin{equation}
\label{eq:3.12}
S\gg x(\log x)^{k}\exp(-C_{2}k^{2})\sum_{\substack{\mathcal{H}=\lbrace h_{1},...,h_{k}\rbrace\ \textrm{admissible}\\ 0< h_{1}<h_{2}<...<h_{k}<\lambda\log x}} 1,
\end{equation}
for a suitable constant $C_{2}>0$. Thus, we are left to obtain a lower bound for the inner sum of \eqref{eq:3.12}.
We greedily sieve the interval $[1,\lambda\log x]$, by removing for each prime $p\leq k$ in turn any elements from the residue class modulo $p$ which contains the fewest elements. The resulting set has size at least
$$\lambda\log x\prod_{p\leq k}\left(1-\frac{1}{p•}\right)\gg \frac{\log x}{k^{4}(\log k)^{3}•},$$
by our choice of $\lambda$ and Mertens's theorem. Now, any choice of $k$ distinct $h_{i}$ from this set will constitute an admissible set $\mathcal{H}$. Therefore, we obtain the lower bound
$$\sum_{\substack{\mathcal{H}=\lbrace h_{1},...,h_{k}\rbrace\ \textrm{admissible}\\ 0< h_{1}<h_{2}<...<h_{k}<\lambda\log x}}1\geq k^{-k}\left(C_{3}\frac{\log x}{k^{4}(\log k)^{3}•}-k\right)^{k}\gg(\log x)^{k}\exp(-C_{4} k^{2}),$$
for certain constants $C_{3}, C_{4}>0$. This leads to $S\gg x(\log x)^{2k}\exp(-C_{5} k^{2}),$ for a suitable absolute constant $C_{5}>0$. We conclude this section finding a lower bound for $|I(x)|$. By combining \eqref{eq:3.7} with the above information on $S$ we obtain
\begin{equation}
\label{eq:3.13}
|I(x)|\gg xk^{-1}\exp(-C_{5}k^{2})\exp(-C_{6}k^{4}(\log k))\gg x\exp(-C_{7}k^{5}),
\end{equation}
for certain absolute constants $C_{6},C_{7}>0$.
\section{Proof of Lemma 3.1}
\begin{proof}[Proof of Lemma 3.1] Let $k$ be a fixed number. Since $\Delta_{\mathcal{L}}=|g|^{k+1}\prod_{i=1}^{k}|h-h_{i}|$ and $\Delta_{\mathcal{L}}/\varphi(\Delta_{\mathcal{L}})=\sum_{d|\Delta_{\mathcal{L}}} \mu^{2}(d)/\varphi(d)$, we may write
\begin{equation}
\label{eq:4.1}
\sum_{\substack{h\leq 5\lambda\log x\\ (h,g)=1\\ gn+h\not\in \mathcal{L}}}\frac{\Delta_{\mathcal{L}}}{\varphi(\Delta_{\mathcal{L}•})}=\frac{g}{\varphi(g)•}\sum_{\substack{h\leq 5\lambda\log x\\ (h,g)=1\\ gn+h\not\in \mathcal{L}}}\sum_{\substack{d|\Delta_{\mathcal{L}}\\ (d,g)=1}}\frac{\mu^{2}(d)}{\varphi(d)•}
\end{equation}
$$\ll \frac{g}{\varphi(g)•}\sum_{\substack{h\leq 5\lambda\log x\\ (h,g)=1\\ gn+h\not\in \mathcal{L}}}\bigg(\sum_{\substack{1\leq d\leq \sqrt{\lambda\log x}\\ d|\Delta_{\mathcal{L}}\\ (d,g)=1}}\frac{\mu^{2}(d)}{\varphi(d)•}+\sum_{\substack{d>\sqrt{\lambda\log x}\\ d|\Delta_{\mathcal{L}}}}\frac{\mu^{2}(d)}{\varphi(d)•}\frac{\sum_{p|d}\log p}{\log(\sqrt{\lambda\log x•})•}\bigg).$$
Regarding the second term in parenthesis we use $\frac{g}{\varphi(g)•}\ll \log\log g\ll \log\log\log x$, which holds by \cite[$\S 5.4$, Theorem 4]{T}, and it becomes
\begin{equation}
\label{eq:4.2}
\ll(\log\log\log x)\sum_{\substack{h\leq 5\lambda\log x\\ gn+h\not\in \mathcal{L}}}\sum_{p|\Delta_{\mathcal{L}}}\frac{\log p}{p\log(\sqrt{\lambda\log x•})•}\frac{\Delta_{\mathcal{L}}}{\varphi(\Delta_{\mathcal{L}})•}.
\end{equation}
Since $\frac{\Delta_{\mathcal{L}}}{\varphi(\Delta_{\mathcal{L}})•}\ll \log\log \Delta_{\mathcal{L}}\ll\log\log\log x$, again by \cite[$\S 5.4$, Theorem 4]{T} and the fact that $g,h,h_{1},...,h_{k}$ $\ll\log x$, and using Mertens's theorem to evaluate   
$$\sum_{p|\Delta_{\mathcal{L}}}\frac{\log p}{p•}\leq \sum_{p\leq \omega(\Delta_{\mathcal{L}})}\frac{\log p}{p•} \ll \log(\omega(\Delta_{\mathcal{L}}))\ll \log\log\log x,$$ 
the second term in the last line of \eqref{eq:4.1} becomes
\begin{equation}
\label{eq:4.3}
\ll \frac{(\log\log\log x)^{3}}{\log\log x•}\lambda\log x\ll \lambda(\log x)(\log k).
\end{equation}
Now, we concentrate on the first term. It is equal to
\begin{equation}
\label{eq:4.4}
\frac{g}{\varphi(g)•}\sum_{\substack{d\leq \sqrt{\lambda\log x}\\ (d,g)=1}}\frac{\mu^{2}(d)}{\varphi(d)•}\sum_{\substack{h\leq 5\lambda\log x\\ h\neq h_{1},..,h_{k}\\(h,g)=1\\d|\Delta_{\mathcal{L}}}}1.
\end{equation}
The innermost sum may be written as
\begin{equation}
\label{eq:4.5}
\sum_{\substack{1\leq c\leq d\\ P(c)\equiv 0\pmod{d}}}\sum_{\substack{h\leq 5\lambda\log x\\ h\neq h_{1},..,h_{k}\\(h,g)=1\\h\equiv c\pmod{d}}}1,
\end{equation}
where we let $P(x)$ to be the polynomial $P(x)=\prod_{i=1}^{k}(x-h_{i})$. We put $\rho(d)$ the number of solutions of $P(x) \pmod{d}$. We will use the following Lemma.
\begin{lem}
\label{lem:4.1} Let us consider $h_{1},...,h_{k},g, \lambda$ as above. Suppose that $d\leq \sqrt{\lambda\log x}$ is a positive integer coprime with $g$. We have
\begin{equation}
\label{eq:4.6}
\sum_{\substack{h\leq 5\lambda\log x\\ h\neq h_{1},..,h_{k}\\(h,g)=1\\h\equiv c\pmod{d}}}1\ll \frac{\varphi(g)•}{g}\frac{\lambda\log x}{d}, 
\end{equation}
for every residue class $c\pmod{d}$.
\end{lem}
Now, we see how this implies the Lemma \ref{lem:3.1}. Inserting \eqref{eq:4.6} into \eqref{eq:4.5} and \eqref{eq:4.5} into \eqref{eq:4.4} we find
\begin{equation}
\label{eq:4.7}
\frac{g}{\varphi(g)•}\sum_{\substack{d\leq \sqrt{\lambda\log x}\\ (d,g)=1}}\frac{\mu^{2}(d)}{\varphi(d)•}\sum_{\substack{h\leq 5\lambda\log x\\ h\neq h_{1},..,h_{k}\\(h,g)=1\\d|\Delta_{\mathcal{L}}}}1\ll \lambda(\log x)\sum_{d\leq \sqrt{\lambda\log x}}\frac{\mu^{2}(d)\rho(d)}{d\varphi(d)•}\ll \lambda(\log x)\prod_{p\leq \lambda\log x}\left(1+\frac{\rho(p)}{p(p-1)}\right)
\end{equation}
$$\ll \lambda(\log x)\prod_{p\leq k}\left(1+\frac{1}{p-1}\right)\prod_{p> k}\left(1+\frac{k}{p(p-1)}\right)\ll \lambda(\log x)(\log k),$$
because the second product converges and the first one is $\ll\log k$ by Mertens's theorem.
Collecting our results, these give the upper bound \eqref{eq:3.9}.
\end{proof}
We now return to prove Lemma \ref{lem:4.1}.
\begin{proof}[Proof of Lemma 4.1]
By the Selberg's upper bound \cite[Theorem 7.1]{FI}, we have
\begin{equation}
\label{eq:4.8}
\sum_{\substack{h\leq 5\lambda\log x\\ h\equiv c\pmod{d}\\ p|h\Rightarrow p\nmid g}}1\leq \frac{\lambda\log x}{Jd}+O\bigg(\sum_{\substack{e\leq D\\ p|e\Rightarrow p|g}} 3^{\omega(e)}\mu^{2}(e)\bigg),
\end{equation}
where 
\begin{equation}
\label{eq:4.9}
J=\sum_{\substack{e\leq \sqrt{D}\\ p|e\Rightarrow p|g}}\frac{\mu^{2}(e)}{\varphi(e)}\geq \sum_{e|g}\frac{\mu^{2}(e)}{\varphi(e)}=\frac{g}{\varphi(g)},
\end{equation}
taking $D=4\log^{2}x$, so that $g\leq \sqrt{D}$.
Finally, with this choice of $D$ we find
\begin{equation}
\label{eq:4.10}
\sum_{\substack{e\leq D\\ p|e\Rightarrow p|g}} 3^{\omega(e)}\mu^{2}(e)\leq \prod_{p|g}4=4^{\omega(g)}\leq \exp\left( C\frac{\log g}{\log\log g}\right)
\end{equation}
using \cite[$\S 5.3$, Theorem 3]{T}, for a certain absolute constant $C>0$. By our definition of $g$ we may bound the error term in \eqref{eq:4.8} with
\begin{equation}
\label{eq:4.11}
\sum_{\substack{e\leq D\\ p|e\Rightarrow p|g}} 3^{\omega(e)}\mu^{2}(e)\leq \exp\left( C'\frac{\log\log x}{\log\log\log x}\right)\ll \frac{\sqrt{\lambda\log x}}{(\log\log\log x)^{2}}\ll \frac{\lambda \log x}{d•}\frac{\varphi(g)•}{g},
\end{equation}
by \cite[$\S 5.4$, Theorem 4]{T}, for a suitable $C'>0$, when $x$ is sufficiently large (in terms of $k$ also). Inserting \eqref{eq:4.11} and \eqref{eq:4.9} into \eqref{eq:4.8} we find the thesis of the Lemma.
\end{proof}
\section{The combinatorial construction}
Consider an interval $I\in I(x)$. There exist an integer $x<n\leq 2x$ and an admissible set $\mathcal{H}=\lbrace h_{1},...,h_{k}\rbrace$, with  $0< h_{1}<h_{2}<...<h_{k}< \lambda\log x$, such that $I=[gn, gn+5\lambda\log x]$ and 
$$|[gn,gn+5\lambda\log x]\cap \mathbb{P}|=|\lbrace gn+h_{1},...,gn+h_{k}\rbrace\cap\mathbb{P}|\geq m.$$
Let us define 
\begin{equation}
\label{eq:5.1}
I_{j}=[N_{j}, N_{j}+\lambda \log N_{j}],\ \ N_{j}=gn+j,
\end{equation}
for $j=0,...,\lfloor \lambda\log N_{0}\rfloor$. For any such $j$ we have $I_{j}\subseteq I$. Indeed, by \eqref{eq:5.1} it is sufficient to prove that 
$$gn\leq N_ {j}<N_{j}+\lambda \log N_{j}<gn+5\lambda\log x,$$
which holds because 
$$j+\lambda \log N_{j}<\lambda\log N_{0}+\lambda \log N_{j}<\lambda\log N_{0}+\lambda\log 2N_{0}<3\lambda\log N_{0}<5\lambda\log x,$$
if $x$ is sufficiently large in terms of $k$. In particular, we have
$$I_{j}\cap \mathbb{P}=I_{j}\cap \lbrace gn+h_{1},...,gn+h_{k}\rbrace\cap\mathbb{P}.$$
For the particular value $j=h_{1}$ we find that 
$$I_{j}\cap \lbrace gn+h_{1},...,gn+h_{k}\rbrace=\lbrace gn+h_{1},...,gn+h_{k}\rbrace,$$
because $gn+h_{1}=N_{h_{1}}$ and 
$$gn+h_{k}<N_{h_{1}}+\lambda \log N_{h_{1}},$$
since 
$$h_{k}-h_{1}<\lambda\log x<\lambda \log N_{h_{1}}.$$
Finally, for the special value $j=\lfloor \lambda\log N_{0}\rfloor$ we have 
$$I_{j}\cap\lbrace gn+h_{1},...,gn+h_{k}\rbrace=\emptyset.$$
In fact, by \eqref{eq:5.1} it is sufficient to show that $gn+h_{k}<N_{j}$, which holds because $h_{k}<j$. Indeed, 
$$j=\lfloor \lambda\log N_{0}\rfloor> \lambda\log N_{0}-1>\lambda\log(x\log x)-1=\lambda\log x+\lambda\log\log x -1>\lambda\log x>h_{k},$$
if $x$ is sufficiently large in terms of $k$. Thus, we have proved that
$$|I_{h_{1}}\cap \mathbb{P}|\geq m, |I_{\lfloor \lambda\log N_{0}\rfloor}\cap\mathbb{P}|=0.$$
But if $|I_{j}\cap \mathbb{P}|<|I_{j+1}\cap\mathbb{P}|$, then $|I_{j+1}\cap\mathbb{P}|=|I_{j}\cap\mathbb{P}|+1$ and this means that there must exist some $j\in\lbrace 0,...,\lfloor \lambda\log N_{0}\rfloor\rbrace$ for which $|I_{j}\cap\mathbb{P}|=m$. Again by the particular choice of our intervals \eqref{eq:5.1}, it is clear that in this way we can construct an injective correspondence between intervals in $I(x)$ and intervals of the form $[N, N+\lambda\log N]$ containing exactly $m$ primes, with $N<4x\log x+\lambda\log (4x\log x)<5x\log x$, if $x$ is sufficiently large. By the lower bound \eqref{eq:3.13}, we find that for every $m\geq 0$ and for each $\lambda\leq\varepsilon$, with $\varepsilon$ a small multiple of $k^{-4}(\log k)^{-2}$,
\begin{equation}
\label{eq:5.2}
|\lbrace N\leq 5x\log x: |[N, N+\lambda\log N]\cap\mathbb{P}|=m\rbrace|\gg x\exp(-C_{7}k^{5}),
\end{equation}
when $x$ is sufficiently large, which is equivalent to
\begin{equation}
\label{eq:5.3}
|\lbrace N\leq X: |[N, N+\lambda\log N]\cap\mathbb{P}|=m\rbrace|\gg \frac{X}{\log X•}\exp(-C_{7}k^{5}),
\end{equation}
when $X$ is sufficiently large in terms of $k$, which proves Theorem \ref{thm:1.2}. 
\section{Concluding remarks} In this section we make explicit the dependence between $\lambda$ and $m$. We know that $\lambda\ll k^{-4}(\log k)^{-2}$ and that $k\gg \exp(48m)$. Therefore, we certainly have $\lambda (48m)^{2}\exp(192m)\ll 1.$ Moreover, the explicit constant $\exp(-C_{7}k^{5})$ in \eqref{eq:5.3} can be written in terms of $m$ as $\exp(-C_{8}\exp(240m))$, for a suitable absolute constant $C_{8}>0$. We remark that in \cite[Section 5]{F} Freiberg pointed out the possibility of an improvement of Theorem \ref{thm:1.1}, for certain values of $\lambda$ and $m$ satisfying an interdependence similar to the previous one, via a proof that uses Maynard's sieve alone and does not involve the Erd\H{o}s--Rankin construction. The present paper goes in this direction. Finally, we add a note about Lemma \ref{lem:3.1}. It is a generalization of \cite[Lemma 8.1 (ii)]{M1} to admissible sets of linear functions $\mathcal{L}=\lbrace gn+h_{1},...,gn+h_{k}\rbrace$, when $g\leq 2\log x$. Actually, we can extend Lemma \ref{lem:3.1} to all the sets $\mathcal{L}$ in which $g\ll (\log x)^{O(1)}$, because its proof works also in this case, or in which $g/\varphi(g)\ll 1$, inspecting the proof of \cite[Lemma 8.1 (ii)]{M1}. In this more general setup, it is used implicitly in \cite[Theorem 3.3]{M1} and we made it explicit here.
\section{Acknowledgements}
I would like to thank Dimitris Koukoulopoulos for suggesting this problem and James Maynard for useful comments and encouragements.

\end{document}